\documentclass[11pt,a4paper, reqno]{amsart}

\usepackage{amssymb,mathrsfs,enumitem,amsmath,amsthm,bbold}
\usepackage[mathscr]{euscript}
\usepackage{color}

\usepackage[all]{xy}
\definecolor{Chocolat}{rgb}{0.36, 0.2, 0.09}
\definecolor{BleuTresFonce}{rgb}{0.215, 0.215, 0.36}
\definecolor{EgyptianBlue}{rgb}{0.06, 0.2, 0.65}

\usepackage[OT2,OT1]{fontenc}

\newcommand\cyrillic[1]{
	{\fontencoding{OT2}\fontfamily{wncyr}\selectfont #1}
		}

\newcommand\mathcyr[1]{\text{\cyrillic{#1}}}
\newcommand\Sha{\textnormal{\mathcyr{Sh}}} 

\usepackage[colorlinks,final,hyperindex]{hyperref}
\hypersetup{citecolor=BleuTresFonce, urlcolor=EgyptianBlue, linkcolor=Chocolat}

\usepackage{fourier}

\newtheorem*{itheorem}{Theorem}
\newtheorem{theorem}{Theorem}[section]

\newtheorem{lemma}[theorem]{Lemma}
\newtheorem{proposition}[theorem]{Proposition}

\theoremstyle{definition}

\DeclareMathAlphabet{\pazocal}{OMS}{zplm}{m}{n}

\def\calO{\pazocal{O}}

\def\calS{\pazocal{S}}
\def\calT{\pazocal{T}}

\def\calX{\pazocal{X}}

\DeclareMathAlphabet{\mathbbold}{U}{bbold}{m}{n}

\def\k{\mathbbold{k}}

\newcommand{\twocor}[3]{\ensuremath{ 
\vcenter{\hbox{\xymatrix@R=.4pc@C=.2pc{ 
    #2\ar@{-}[dr] & &#3\ar@{-}[dl]\\
	&*+[o][F-]{#1}\ar@{-}[d]&\\
	&*{}&
}}}}}

\newcommand{\lbincomb}[5]{\ensuremath{
\vcenter{\hbox{\xymatrix@R=.4pc@C=.2pc{ 
			#3\ar@{-}[dr] &&#4\ar@{-}[dl] &\\
			&*+[o][F-]{#2}\ar@{-}[dr]&&#5\ar@{-}[dl]\\
			&&*+[o][F-]{#1}\ar@{-}[d]&\\
			&&*{}&
}}}}}
	
\newcommand{\rbincomb}[5]{\ensuremath{
\vcenter{\hbox{\xymatrix@R=.4pc@C=.2pc{ 
			 &#4\ar@{-}[dr] &&#5\ar@{-}[dl] \\
			#3\ar@{-}[dr]&&*+[o][F-]{#2}\ar@{-}[dl]\\
			&*+[o][F-]{#1}\ar@{-}[d]&&\\
			&*{}&&
}}}}}

\newcommand{\binfork}[7]{\!\!\!\ensuremath{
\vcenter{\hbox{\xymatrix@R=.4pc@C=.2pc{ 
			#4\ar@{-}[dr] && #5\ar@{-}[dl]\,\,#6\ar@{-}[dr]&&#7\ar@{-}[dl]\\
			&*+[o][F-]{#2}\ar@{-}[dr] &   &*+[o][F-]{#3}\ar@{-}[dl]\\
			&&*+[o][F-]{#1}\ar@{-}[d]&&\\
			&&*{}&&
		}}}}\!\!\!\!}

\newcommand{\llbincomb}[7]{\ensuremath{\!\!\!\!
 \vcenter{\hbox{\xymatrix@R=.4pc@C=.2pc{ 
 			#4\ar@{-}[dr] && #5\ar@{-}[dl]&&\\
 			&*+[o][F-]{#3}\ar@{-}[dr]&&#6\ar@{-}[dl] \\
 			&&*+[o][F-]{#2}\ar@{-}[dr] &   &#7\ar@{-}[dl]\\
 			&&&*+[o][F-]{#1}\ar@{-}[d]&\\
 			&&&*{}&
 		}}}}}

\newcommand{\lbtcomb}[6]{\ensuremath{
\vcenter{\hbox{\xymatrix@R=.4pc@C=.2pc{ 
			#3\ar@{-}[dr] &#4\ar@{-}[d]& #5\ar@{-}[dl]&\\
			&*+[o][F-]{#2}\ar@{-}[dr]&&#6\ar@{-}[dl]\\
			&&*+[o][F-]{#1}\ar@{-}[d]&\\
			&&*{}&
}}}}}

\newcommand{\ltbcomb}[6]{\ensuremath{
\vcenter{\hbox{\xymatrix@R=.4pc@C=.2pc{ 
			#3\ar@{-}[dr] &&#4\ar@{-}[dl] &\\
			&*+[o][F-]{#2}\ar@{-}[dr]&#5\ar@{-}[d]&#6\ar@{-}[dl]\\
			&&*+[o][F-]{#1}\ar@{-}[d]&\\
			&&*{}&
}}}}}

\newcommand{\rrbincombfork}[9]{\ensuremath{
\vcenter{\hbox{\xymatrix@R=.4pc@C=.2pc{ 
			&#6\ar@{-}[dr] && #7\ar@{-}[dl]&#8\ar@{-}[dr] && #9\ar@{-}[dl]\\
			&&*+[o][F-]{#5}\ar@{-}[dr]& && *+[o][F-]{#3}\ar@{-}[dll]&\\
			&#4\ar@{-}[dr] &   &*+[o][F-]{#2}\ar@{-}[dl]&&\\
			&&*+[o][F-]{#1}\ar@{-}[d]&&&\\
			&&*{}&&&&
}}}}\!\!\!\!\!\!\!\!\!\!}

\begin{document}

\title[Characterisation of Lie algebras]{A characterisation of Lie algebras\\ using ideals and subalgebras}

\author{Vladimir Dotsenko}

\address{ 
Institut de Recherche Math\'ematique Avanc\'ee, UMR 7501, Universit\'e de Strasbourg et CNRS, 7 rue Ren\'e-Descartes, 67000 Strasbourg CEDEX, France}

\email{vdotsenko@unistra.fr}

\author{Xabier Garc\'ia-Mart\'inez}

\address{CITMAga \& Universidade de Vigo, Departamento de Matem\'aticas, Esc.\ Sup.\ de Enx.\ Inform\'atica, Campus de Ourense, E--32004 Ourense, Spain---Faculty of Engineering Vrije Universiteit Brussel, Pleinlaan 2, B--1050 Brussel, Belgium}

\email{xabier.garcia.martinez@uvigo.gal}

\date{}

\begin{abstract}
We prove that if, for a nontrivial variety of non-associative algebras, every subalgebra of every free algebra is free and $I^2$ is an ideal whenever $I$ is an ideal, then this variety coincides with the variety of all Lie algebras. 
\end{abstract}

\maketitle


\section{Introduction}

The first categorical characterisation of Lie algebras among all varieties of non-associative algebras appeared in~\cite{MR3872845}, via the admissibility of algebraic exponents in the sense of Gray~\cite{MR2925888,MR2990906}. More precisely, the variety of Lie algebras is the unique non-trivial variety of non-associative algebras which is \emph{locally algebraically cartesian closed} (LACC for short), condition that can be interpreted as follows: a variety $\mathfrak{M}$ is LACC if and only if for any algebra $B$ of the variety, the forgetful functor from the category of $B$-actions to $\mathfrak{M}$ has a right adjoint. Another categorical characterisation was obtained in~\cite{MR4330276} where it is shown that the variety of Lie algebras is the unique non-trivial variety of non-associative algebras whose representations functor is representable. In this paper, we give another characterisation which, while relies on categorical methods, only imposes constraints in the language of classical ring theory. Specifically, we prove the following result. 

\begin{itheorem}
Suppose that $\mathfrak{M}$ is a non-trivial variety of non-associative algebras over a field of zero characteristic satisfying the following two conditions:
\begin{itemize}
\item every subalgebra of every free algebra is free
\item for every ideal $I$ in every algebra, $I^2$ is also an ideal
\end{itemize}	
is the variety of Lie algebras.
\end{itheorem}

Both of these properties have been well studied by ring theorists. The first of them, often referred to as the Nielsen--Schreier property, was first established for groups by Nielsen~\cite{MR1512188} and Schreier~\cite{MR3069472}, and later studied extensively for $\k$-linear case, see, e.g., \cite{MR0020986,MR0059892,MR0062112,MR77525}. Recently, methods of operad theory were used to give an effective combinatorial criterion for the Nielsen--Schreier property~\cite{DU22}, which led to infinitely many new examples of Nielsen--Schreier varieties of non-associative algebras. The second property, known as the 2-variety property, goes back to the work of Anderson~\cite{MR285564} and Zwier~\cite{MR281763}; it was further studied by several authors, particularly in the context of defining radicals of algebras of certain varieties, see~\cite{MR469986,MR414640,MR653895,MR797659,MR506474,MR283034}. 

Like the characterisations of the variety of Lie algebras in \cite{MR4330276,MR3872845}, our work uses computer algebra, though in a significantly different way. Comparing the thus obtained characterisations does however lead to an intriguing open problem. It is known that the LACC condition implies that the canonically induced morphisms
 \[
(B \flat X + B \flat Y) \to B \flat (X + Y) 
 \]
(where $B, X, Y$ are free objects and where $B\flat X$ is the kernel of the unique map $B + X \to B$ induced by the identity and zero morphisms, respectively) are isomorphisms. The surjectivity of these maps is equivalent to the category being algebraically coherent~\cite{MR3438233}, which in turn is equivalent to the 2-variety property~\cite{MR3955044}. Our result prompt a natural question as to whether the injectivity of these maps is equivalent to the Nielsen--Schreier property. 

This manuscript is organised as follows. In Section~\ref{sec:recoll} the necessary theoretical background will be recalled. In Section~\ref{sec:analysis} a preliminary analysis will be provided, understanding the two conditions from an operadic perspective. They both provide bounds on the dimensions of components of the operad encoding the given variety; those bounds overlap in a very narrow way. Finally, in Section~\ref{sec:proof} we shall focus on the computational aspect of the proof and exclude most of the potential candidates, proving the main result, Theorem~\ref{th:main}.

\section{Conventions and recollections}\label{sec:recoll}

All algebras considered in this paper are defined over a ground field $\k$ of zero characteristic. Unless otherwise specified, we use the word ``variety'' in the sense of universal algebra: for us, a \emph{variety of algebras} is an equational theory. It is important to not conflate this notion with that of a \emph{variety of algebra structures} on a certain object, which itself can be a subject of extensive study.
Throughout this paper, our main focus is on varieties of non-associative algebras, meaning that each algebra $V$ of each variety of algebras considered has just one structure operaton, a binary product $V\otimes V\to V$. For the recollection of Gr\"obner bases for operads in this section, we offer a much more general context: we only assume that the signature of our variety does not include constants (structure operations of arity $0$) or structure operations of arity $1$. Operadically, these assumptions are described by the word ``reduced'' and ``connected'' respectively.

\subsection{Varieties and symmetric operads}

It is well known that over a field of characteristic zero every system of algebraic identities is equivalent to multilinear ones. This means that all information about a variety of algebras $\mathfrak{M}$ is captured by the collection
 \[
\calO=\calO_\mathfrak{M}:=\{\calO(n)\}_{n\ge 1},
 \]
where $\calO(n)$ is the $S_n$-module of multilinear elements (that is, elements of multidegree $(1,1,\ldots,1)$) in the free algebra $F_\mathfrak{M}\langle x_1,\ldots,x_n\rangle$. This collection of $S_n$-modules has a very rich structure arising from substituting multilinear elements into one another. A clean, even if slightly abstract way to introduce this structure uses the language of linear species, which we shall now recall. 

The theory of species of structures originated at the concept of a combinatorial species, invented by Joyal \cite{MR633783} and presented in great detail in \cite{MR1629341}. The same definitions apply if one changes the target symmetric monoidal category; in particular, if one considers the category of vector spaces, one obtains what is called a linear species. Let us recall some key definitions, referring the reader to~\cite{MR2724388} for further information. 

A \emph{linear species} is a contravariant functor from the groupoid of finite sets (the category whose objects are finite sets and whose morphisms are bijections) to the category of vector spaces. This definition is not easy to digest at a first glance, and a reader with intuition coming from varieties of algebras is invited to think of the value $\calS(I)$ of a linear species $\calS$ on a finite set $I$ as of the set of multilinear operations of type $\calS$ (accepting arguments from some vector space $V_1$ and assuming values in some vector space $V_2$) whose inputs are indexed by $I$. A linear species $\calS$ is said to be \emph{reduced} if $\calS(\varnothing)=0$; this means that we do not consider ``constant'' multilinear operations. (This is perhaps the only situation where several different terminologies clash in our paper: we use the word ``reduced'' for linear species to indicate that the value on the empty set is zero, and for Gr\"obner bases to indicate that we consider the unique Gr\"obner basis of a certain irreducible form.)

The \emph{composition product} of linear species is defined by the formula
 \[
(\calS_1\circ\calS_2)(I)
=\bigoplus_{n\ge 0}\calS_1(\{1,\ldots,n\})\otimes_{\k S_n}\left(\bigoplus_{I=I_1\sqcup \cdots\sqcup I_n}\calS_2(I_1)\otimes\cdots\otimes \calS_2(I_n)\right).
 \]
The linear species $\mathbbold{1}$ which vanishes on a finite set $I$ unless $|I|=1$, and whose value on $I=\{a\}$ is given by $\k a$ is the unit for the composition product: we have $\mathbbold{1}\circ\calS=\calS\circ\mathbbold{1}=\calS$.

Formally, a \emph{symmetric operad} is a monoid with respect to the composition product. It is just the multilinear version of substitution schemes of free algebras discussed above, but re-packaged in a certain way. The advantage is that the existing intuition of monoids and modules over them, available in any monoidal category \cite{MR0354798}, can be used for studying varieties of algebras. 

The free symmetric operad generated by a linear species $\calX$ is defined as follows. Its underlying linear species is the species $\calT(\calX)$ for which $\calT(\calX)(I)$ is spanned by decorated rooted trees (including the rooted tree without internal vertices and with just one leaf, which corresponds to the unit of the operad): the leaves of a tree must be in bijection with $I$, and each internal vertex $v$ of a tree must be decorated by an element of $\calX(I_v)$, where $I_v$ is the set of incoming edges of $v$. Such decorated trees should be thought of as tensors: they are linear in each vertex decoration. The operad structure is given by grafting of trees onto each other. We remark that one can also talk about the free operad generated by a collection of $S_n$-modules, but the formulas will become heavier.

\subsection{Shuffle operads and Gr\"obner bases}

We shall now recall how to develop a workable theory of normal forms in operads using the theory of Gr\"obner bases developed by the first author and Khoroshkin \cite{MR2667136}. It is important to emphasise that it is in general extremely hard to find convenient normal forms in free algebras for a given variety~$\mathfrak{M}$. However, focusing on multilinear elements simplifies the situation quite drastically: for instance, for a basis in multilinear elements for the operad controlling Lie algebras one may take all left-normed commutators of the form
$[[[a_1,a_{i_2}],\cdots],a_{i_n}]$, 
where $i_2$,\ldots, $i_n$ is a permutation of $2$,\ldots,$n$; by contrast, all known bases in free Lie algebras are noticeably harder to describe.

To define Gr\"obner bases for operads, one builds, step by step, an analogue of the theory of Gr\"obner bases for noncommutative associative algebras. To do this, one has to abandon the universe that has symmetries, otherwise there is not even a good notion of a monomial that leads to a workable theory. The kind of monoids that have a good theory of Gr\"obner bases are \emph{shuffle operads}. A rigorous definition of a shuffle operad uses ordered species \cite{MR1629341}, which we shall now discuss in the linear context. 

An \emph{ordered linear species} is a contravariant functor from the groupoid of finite ordered sets (the category whose objects are finite totally ordered sets and whose morphisms are order preserving bijections) to the category of vector spaces. In terms of the intuition with multilinear maps, this more or less corresponds to choosing a basis of multilinear operations whose inputs are indexed by an ordered set $I$. An ordered linear species $\calS$ is said to be \emph{reduced} if $\calS(\varnothing)=0$. 

The \emph{shuffle composition product} of two reduced ordered linear species $\calS_1$ and~$\calS_2$ is defined by the formula
 \[
(\calS_1\circ_\Sha\calS_2)(I)=\bigoplus_{n\ge 1}\calS_1(\{1,\ldots,n\})\otimes\left(\bigoplus_{\substack{I=I_1\sqcup \cdots\sqcup I_n,\\ 
I_1,\ldots, I_n\ne\varnothing,\\ \min(I_1)<\cdots<\min(I_n)}}\calS_2(I_1)\otimes\cdots\otimes \calS_2(I_n)\right),
 \]
and the linear species $\mathbbold{1}$ discussed above may be regarded as an ordered linear species; as such, it is the unit of the shuffle composition product.

Formally, a \emph{shuffle operad} is a monoid with respect to the shuffle composition product. As we shall see below, each symmetric operad gives rise to a shuffle operad, and that is the main reason to care about shuffle operads. However, we start with explaining how to develop a theory of Gr\"obner bases of ideals in free shuffle operads.

To describe free shuffle operads, we first define shuffle trees. Combinatorially, a \emph{shuffle tree} is a planar rooted tree whose leaves are indexed by a finite ordered set $I$ in such a way that the following ``local increasing condition'' is satisfied: for every vertex of the tree, the minimal leaves of trees grafted at that vertex increase from the left to the right. The free shuffle operad generated by an ordered linear species $\calX$ can be defined as follows. It is an ordered linear species $\calT_\Sha(\calX)$ for which $\calT_\Sha(\calX)(I)$ is spanned by decorated shuffle trees: each internal vertex $v$ of a tree must be decorated by an element of $\calX(I_v)$, where $I_v$ is the set of incoming edges of $v$, ordered from the left to the right according to the planar structure. Such decorated trees should be thought of as tensors: they are linear in each vertex decoration. The operad structure is given by grafting of trees onto each other. There are two particular classes of shuffle trees that will be useful for us, the left combs and the right combs. If, for each internal vertex of a shuffle tree, the only input that is not necessarily a leaf is the leftmost one, the tree is called a left comb; similarly, if the only input that is not necessarily a leaf is the rightmost one, the tree is called a right comb. Despite the similar definitions, the two types of combs are quite different combinatorially, for instance if our shuffle operad is generated by one binary operation, there are two left combs with three leaves but only one right comb, all displayed in the following figure:
 \[
\lbincomb{}{}{1}{2}{3},\quad \lbincomb{}{}{1}{3}{2}, \quad\rbincomb{}{}{1}{2}{3} .
 \]

Given a basis of the vector space of an ordered linear species $\calX$, one may consider all shuffle trees whose vertices are decorated by those basis elements. Such shuffle trees with leaves in a bijection with the given ordered set $I$ form a basis of $\calT_\Sha(\calX)(I)$, and we shall think of them as monomials in the free shuffle operad. 

The next step in developing a theory of Gr\"obner bases is to define divisibility of monomials. Suppose that we have a shuffle tree $S$. We can insert another shuffle tree $S'$ into an internal vertex of $S$, and connect its leaves to the children of that vertex so that the order of leaves agrees with the left-to-right order of the children. We say that the thus obtained shuffle tree is divisible by $S'$, and use this notion of divisibility to define divisibility of decorated shuffle trees, that is of monomials in the free operad. 
The key feature of divisibility that we shall use in most of our proofs is that right combs are very ``rare'': for each sequence of labels of internal vertices, there is a unique right comb with that sequence, and consequently, divisibility by a right comb is extremely easy to check (the condition on the order of leaves is vacuous).

Once divisibility is understood, the usual Gr\"obner--Shirshov method of computing S-polynomials (in the language of Shirshov, one would say ``compositions'', which has the huge disadvantage in the case of operads where the same word is used to talk about the monoid structure), normal forms, etc.\ works in the usual way. The only other required ingredient is an \emph{admissible ordering of monomials}, that is a total ordering of shuffle trees with the given set of leaf labels which is compatible with the shuffle operad structure. Such orderings exist, and we invite the reader to consult \cite{MR3642294,MR4114993} for definitions and examples. For us the so called graded path-lexicographic ordering and reverse graded path-lexicographic ordering will be of particular importance. With respect to the former, the trees are first compared by the depth of their leaves, while with respect to the latter, one reverses the comparison with respect to the depth of the leaves (in both cases, leaves are considered one by one in their given order).

Note that there is a forgetful functor $\calS\mapsto \calS^f$ from all linear species to ordered linear species; it is defined by the formula $\calS^f(I):=\calS(I^f)$, where $I$ is a finite totally ordered set and $I^f$ is the same set but with the total order ignored. The reason to consider ordered linear species and shuffle operads is explained by the following proposition. 

\begin{proposition}[{\cite{MR3642294,MR2667136}}]
For any two linear species $\calS_1$ and $\calS_2$, we have the ordered linear species isomorphism
\[
(\calS_1\circ\calS_2)^f\cong\calS_1^f\circ_\Sha\calS_2^f.
\]
In particular, applying the forgetful functor to a reduced symmetric operad gives a shuffle operad. 
\end{proposition}

This result shows that the forgetful functor from symmetric operads to shuffle operads allows one to go from the universe of ``interesting'' objects (actual varieties of algebras) to the universe of ``manageable'' objects (shuffle operads) without losing much information (just the symmetric group actions end up ignored); in particular, one can determine bases and dimensions of components of an operad, which is crucial for the main result of this paper.

\section{Preliminary analysis}\label{sec:analysis}

In this section, we establish the following result which will then be used to analyse our problem using computer algebra. 

\begin{proposition}\label{prop:BoundsMeet}
Let $\mathfrak{M}$ be a Nielsen--Schreier 2-variety of non-associative algebras encoded by an operad $\calO$. One of the following possibilities occurs:
\begin{itemize}
\item the vector space $\calO(2)$ is equal to $\{0\}$, and $\mathfrak{M}$ is trivial,
\item the vector space $\calO(2)$ is of dimension $1$, and $\mathfrak{M}$ is the variety of all Lie algebras, 
\item the vector space $\calO(2)$ is of dimension $2$, the module of quadratic relations of $\calO$ is of dimension $4$, and the operad $\calO$ has a quadratic Gr\"obner basis for the reverse path-lexicographic ordering.
\end{itemize}
\end{proposition}

Our strategy is as follows. We shall first recall results of~\cite{DU22} allowing one to give a lower bound on $\dim\calO(n)$ for a Nielsen--Schreier variety. We then establish an upper bound on $\dim\calO(n)$ for a 2-variety. Remarkably, those bounds are exactly the same, which will force the statement of the theorem to hold.

By definition, a variety of non-associative algebras has just one structure operation which is binary. The vector space $\calO(2)$ is the $S_2$-module generated by this operation, which explains the trichotomy in the statement of the theorem: such a module may be of dimension $0$, $1$, or $2$. 

The following result is a part of~\cite[Cor.4.7]{DU22}; it only depends on a small part of \emph{op.\ cit.}, so we include a detailed proof for completeness. 

\begin{lemma}\label{lm:NSBound}
Suppose that $\mathfrak{M}$ is a Nielsen--Schreier variety of algebras whose structure operations are all of arity $2$ and form a $k$-dimensional vector space. Then for the corresponding operad $\calO$, we have $\dim\calO(n)\ge k^{n-1}(n-1)!$.
\end{lemma}

\begin{proof}
Let us denote by $\calX$ the species of generators of $\calO$. Since $\mathfrak{M}$ is Nielsen--Schreier, according to~\cite[Th.~1]{MR1302528}, for each free algebra $A=\calO(V)$, its universal multiplicative enveloping algebra $U_\calO(A)$ is a free associative algebra. We have
 \[
U_\calO(A)\cong\partial(\calO)\circ_{\calO} A=\partial(\calO)\circ_{\calO} \calO(V)\cong\partial(\calO)(V), 
 \] 
with the product of $U_\calO(A)$ induced from that of $\partial(\calO)$ on the twisted associative algebra level, so $\partial(\calO)$ must be free as a twisted associative algebra. Clearly, $\partial(\calX)$ is a part of the minimal generating set of $\partial(\calO)$, and so $\partial(\calX)$ generates a free twisted associative subalgebra. The dimension of the $n$-th component of the free twisted associative algebra generated by a $k$-dimensional species supported on one-element sets is $k^nn!$, and it remains to shift the index by one to account for the application of $\partial$.
\end{proof}

We shall now analyse the 2-variety condition. Recall that it is established in~\cite[p.~30]{MR285564} that a variety of non-associative algebras $\mathfrak{M}$ is a 2-variety if and only if the following two identities are satisfied in each algebra of $\mathfrak{M}$:

\begin{multline}\label{eq:2var1}
(x_1x_2)x_3=\lambda_1(x_3x_1)x_2+\lambda_2(x_1x_3)x_2+\lambda_3x_2(x_3x_1)+\lambda_4x_2(x_1x_3)\\
+\lambda_5(x_3x_2)x_1+\lambda_6(x_2x_3)x_1+\lambda_7x_1(x_3x_2)+\lambda_8x_1(x_2x_3),
\end{multline}

\begin{multline}\label{eq:2var2}
x_3(x_1x_2)=\rho_1(x_3x_1)x_2+\rho_2(x_1x_3)x_2+\rho_3x_2(x_3x_1)+\rho_4x_2(x_1x_3)\\
+\rho_5(x_3x_2)x_1+\rho_6(x_2x_3)x_1+\rho_7x_1(x_3x_2)+\rho_8x_1(x_2x_3).
\end{multline}

We shall use these identities to establish an upper bound on dimensions of components of any operad encoding a 2-variety of non-associative algebras.

\begin{lemma}\label{lm:2VarBound}
Let $\mathfrak{M}$ be a 2-variety of non-associative algebras encoded by an operad $\calO$. One of the following possibilities occurs:
\begin{itemize}
\item the vector space $\calO(2)$ is equal to $\{0\}$, and $\mathfrak{M}$ is trivial,
\item the vector space $\calO(2)$ is of dimension $1$, and one the following may occur:
\begin{itemize}
\item $\dim\calO(3)=2$, and $\dim\calO(n)\le (n-1)!$ for all $n$, 
\item $\dim\calO(3)=1$, and $\dim\calO(n)\le 1$ for all $n$,
\item $\dim\calO(3)=0$,
\end{itemize}
\item the vector space $\calO(2)$ is of dimension $2$, and $\dim\calO(n)\le 2^{n-1}(n-1)!$ for all $n$.
\end{itemize}
\end{lemma}

\begin{proof}
As above, the vector space $\calO(2)$ is the $S_2$-module generated by the only structure operation of $\mathfrak{M}$, so it is of dimension $0$, $1$, or $2$. The 2-variety condition implies that $\calO$ has at least one relation that is quadratic (in the structure operation, so in the more classical language, an identity of degree $3$ holds in all algebras of $\mathfrak{M}$).

If $\dim\calO(2)=0$, our assertion is obvious. 
Suppose that $\dim\calO(2)=1$, and the structure operation is commutative. The space of elements of arity three in the free operad generated by one commutative binary operation is three-dimensional, and as an $S_3$-module, it is the sum of the trivial representation and the two-dimensional irreducible representation. This immediately implies that in each 2-variety of commutative algebras one of the following identities holds:
\begin{itemize}
\item the mock-Lie identity $(a_1a_2)a_3+(a_2a_3)a_1+(a_3a_1)a_2=0$, 
\item the associativity identity $(a_1a_2)a_3=a_1(a_2a_3)$, 
\item the nilpotence identity $(a_1a_2)a_3=0$.
\end{itemize}
For the first of them, the reduced Gr\"obner basis of the corresponding shuffle operad for the reverse path-lexicographic ordering contains the element 
 \[
a_1(a_2a_3)+(a_1a_2)a_3+(a_1a_3)a_2,
 \] 
and the leading term $a_1(a_2a_3)$ of this element eliminates all shuffle trees that are not left combs. Thus, we have $\dim\calO(n)\le(n-1)!$ for all $n$. For the second identity, identity one has $\dim\calO(n)\le 1$ for all $n$, since imposing the associativity condition alone gives us the operad of commutative associative algebras. Finally, for the third identity, one clearly has $\calO(n)=0$ for all $n\ge 3$.

Suppose that $\dim\calO(2)=1$, and the structure operation is anti-commutative. The space of elements of arity three in the free operad generated by one anti-commutative binary operation is three-dimensional, and as an $S_3$-module, it is the sum of the sign representation and the two-dimensional irreducible representation. This immediately implies that in each 2-variety of anti-commutative algebras one of the following identities holds:
\begin{itemize}
\item the Jacobi identity $(a_1a_2)a_3+(a_2a_3)a_1+(a_3a_1)a_2=0$, 
\item the anti-associativity identity $(a_1a_2)a_3+a_1(a_2a_3)=0$, 
\item the nilpotence identity $(a_1a_2)a_3=0$.
\end{itemize}
For the first of them, the reduced Gr\"obner basis of the corresponding shuffle operad for the reverse path-lexicographic ordering contains the element
 \[
a_1(a_2a_3)-(a_1a_2)a_3+(a_1a_3)a_2,
 \]
 and the leading term $a_1(a_2a_3)$ of this element eliminates all shuffle trees that are not left combs. Thus, we have $\dim\calO(n)\le (n-1)!$ for all $n$. For the second identity, an immediate computation shows that $\calO(n)=0$ for all $n\ge 4$. Finally, for the third identity, one clearly has $\calO(n)=0$ for all $n\ge 3$.

It remains to consider the case $\dim\calO(2)=2$. We shall once again examine the reduced Gr\"obner basis of the corresponding shuffle operad for the reverse path-lexicographic ordering. Let us follow the simplest way of describing the corresponding shuffle operad~\cite[Sec.~5.3.4]{MR3642294}, and choose the operations $u(a_1,a_2)=a_1a_2$ and $v(a_1,a_2)=a_2a_1$ as the basis of $\calO(2)$. Then all shuffle trees whose internal vertices are labelled by $\{u,v\}$ form a basis of the corresponding free shuffle operad. Moreover, in the $S_3$-orbit of the identities~\eqref{eq:2var1} and~\eqref{eq:2var2}, we can easily find four linearly independent ones: they correspond to the identities with the left hand sides $x_1(x_2x_3)$, $x_1(x_3x_2)$, $(x_2x_3)x_1$, and $(x_3x_2)x_1$, or, in plain words, identities allowing to ``hide'' the smallest element inside the brackets; clearly, each of these four monomials appears only in its own identity, so the four identities are linearly independent. In the shuffle context, these four monomials are represented by the trees
 \[
\rbincomb{u}{u}{1}{2}{3}, 
\rbincomb{u}{v}{1}{2}{3}, 
\rbincomb{v}{u}{1}{2}{3}, 
\rbincomb{v}{v}{1}{2}{3}, 
 \]
so they are in fact the leading terms of these identities for the reverse path-lexicographic ordering. These leading terms eliminate all shuffle trees that are not left combs, leading to the upper bound $\dim\calO(n)\le 2^{n-1}(n-1)!$ for the dimensions of components of the operad $\calO$. 
\end{proof}

\begin{proof}[Proof of Proposition \ref{prop:BoundsMeet}]
The case of the trivial variety ($\dim\calO(2)=0$) is obvious. 

Suppose that $\dim\calO(2)=1$. In this case, Lemmas~\ref{lm:NSBound} and~\ref{lm:2VarBound} imply that in order for upper and the lower bound to not contradict each other, we must have $\dim\calO(n)=(n-1)!$ for all $n$. This happens if and only if all left combs are linearly independent in the corresponding shuffle operad, meaning that the only relation of the operad must form the reduced Gr\"obner basis for the reverse path-lexicographic ordering.
In the anti-commutative case, the Jacobi identity does indeed form a Gr\"obner basis~\cite[Sec.~5.6.1]{MR3642294}. However, in the commutative case, the element $a_1(a_2a_3)+(a_1a_2)a_3+(a_1a_3)a_2$ alone does not form a Gr\"obner basis: the S-polynomial of this element with itself reduces to the element 
\begin{multline*}
((a_1a_2)a_3)a_4
+((a_1a_2)a_4)a_3
+((a_1a_3)a_2)a_4
\\ +((a_1a_3)a_4)a_2
+((a_1a_4)a_2)a_3
+((a_1a_4)a_3)a_2,
\end{multline*}
which is a nontrivial linear combination of left combs, leading to the strict inequality $\dim\calO(n)<(n-1)!$ for all $n>3$. 

Finally, suppose that $\dim\calO(2)=2$. In this case, Lemmas~\ref{lm:NSBound} and~\ref{lm:2VarBound} imply that in order for upper and the lower bound to not contradict each other, we must have $\dim\calO(n)=2^{n-1}(n-1)!$ for all $n$. Examining the proof of Lemma~\ref{lm:2VarBound}, we note that the upper bound is attained in two steps. First, we have the inequality $\dim\calO(3)\le 8$, and it becomes an equality if and only if the module of quadratic relations of $\calO$ is of dimension $4$. Second, if the latter module is of dimension $4$, we obtain the general upper bound $\dim\calO(n)\le 2^{n-1}(n-1)!$, which is sharp if and only if all left combs are linearly independent in the corresponding shuffle operad, meaning that the four relations whose leading terms are the four possible right combs form the reduced Gr\"obner basis for the reverse path-lexicographic ordering.
\end{proof}

\section{Elimination of the potential candidates}\label{sec:proof}

To finish the proof of the main result, we need to show that no variety of non-associative algebras can meet all the conditions stated in the third bullet point of Proposition~\ref{prop:BoundsMeet}.

\begin{proposition}\label{prop:generalCase}
Let $\mathfrak{M}$ be a $2$-variety of non-associative algebras, encoded by an operad $\calO$. If $\dim \calO(2)=2$ and $\dim \calO(3) = 4$, then the operad $\calO$ cannot have a quadratic Gr\"obner basis for the reverse path-lexicographic order. 
\end{proposition}

\begin{proof}
The strategy of the proof is the following. Any $2$-variety has to satisfy the identities~\eqref{eq:2var1} and~\eqref{eq:2var2} for certain coefficients in $\k$. The condition of having a quadratic Gr\"obner basis imposes some polynomial constraints on those coefficients, and we shall exhibit enough of those constraints to ensure that they cannot be satisfied simultaneously.
	
It will be convenient to use symmetries of operations. For that, we recall what is often referred to as the ``polarisation procedure'' \cite{MR2225770}.	

\begin{lemma}
Let $\mathfrak{M}$ be a variety of non-associative algebras. It is equivalent to a variety of algebras with one commutative and one anticommutative operation.
\end{lemma}

\begin{proof}
Let us consider the operations $a_1\cdot a_2 = a_1a_2 + a_2a_1$ and $a_1 \star a_2 = a_1a_2-a_2a_1$. They are commutative and anticommutative, respectively. Since 
\begin{gather*}
a_1a_2 = \dfrac{1}{2}(a_1\cdot a_2 + a_2 \star a_1),\\
a_2a_1 = \dfrac{1}{2}(a_1\cdot a_2 - a_2 \star a_1),
\end{gather*}
our change of operations is invertible and defines an equivalence of two varieties.
\end{proof}

Now we will examine the $S_3$-module $\calO(3)$. Let us define an ordering of the structure operations by setting $\cdot > \star$. Recall that the reverse path-lexicographic order on the shuffle monomials of degree~$3$ is the following:
	\begin{equation}\label{eq:order}
		\begin{aligned}
			&a_1 \cdot (a_2 \cdot a_3) > a_1 \cdot (a_2 \star a_3) > a_1 \star (a_2 \cdot a_3) > a_1 \star (a_2 \star a_3) \\
			> &(a_1 \cdot a_3) \cdot a_2 > (a_1 \cdot a_2) \cdot a_3 > (a_1 \star a_3) \cdot a_2 > (a_1 \star a_2) \cdot a_3 \\
			> &(a_1 \cdot a_3) \star a_2 > (a_1 \cdot a_2) \star a_3 > (a_1 \star a_3) \star a_2 > (a_1 \star a_2) \star a_3.
		\end{aligned}
	\end{equation}

We note that for the new structure operations, the system of four identities consisting of identities~\eqref{eq:2var1} and~\eqref{eq:2var2} and the identities obtained from them by the action of the transposition $(1 2)\in S_3$ can be rewritten as the system of four identities expressing the right combs
	\[
a_1 \cdot (a_2 \cdot a_3), a_1 \cdot (a_2 \star a_3), a_1 \star (a_2 \cdot a_3), a_1 \star (a_2 \star a_3)
	\]
as linear combinations of left combs. Moreover, the intrinsic commutative and anticommutative character of the operations $(\cdot,\star)$ forces some symmetry and antisymmetry constraints for coefficients of the left combs. Specifically, the four identities must have the form
		\begin{align*}
			a_1 \cdot (a_2 &\cdot a_3) = \alpha_1\big((a_1 \cdot a_3) \cdot a_2 + (a_1 \cdot a_2) \cdot a_3 \big) + \alpha_2\big( (a_1 \star a_3) \cdot a_2 + (a_1 \star a_2) \cdot a_3 \big) \\
			&+ \alpha_3\big( (a_1 \cdot a_3) \star a_2 + (a_1 \cdot a_2) \star a_3 \big) + \alpha_4\big( (a_1 \star a_3) \star a_2 + (a_1 \star a_2) \star a_3 \big), \\
			a_1 \cdot (a_2 &\star a_3) = \beta_1\big((a_1 \cdot a_3) \cdot a_2 - (a_1 \cdot a_2) \cdot a_3 \big) + \beta_2\big( (a_1 \star a_3) \cdot a_2 - (a_1 \star a_2) \cdot a_3 \big) \\
			&+ \beta_3\big( (a_1 \cdot a_3) \star a_2 - (a_1 \cdot a_2) \star a_3 \big) + \beta_4\big( (a_1 \star a_3) \star a_2 - (a_1 \star a_2) \star a_3 \big), \\
			a_1 \star (a_2 &\cdot a_3) = \gamma_1\big((a_1 \cdot a_3) \cdot a_2 + (a_1 \cdot a_2) \cdot a_3 \big) + \gamma_2\big( (a_1 \star a_3) \cdot a_2 + (a_1 \star a_2) \cdot a_3 \big) \\
			&+ \gamma_3\big( (a_1 \cdot a_3) \star a_2 + (a_1 \cdot a_2) \star a_3 \big) + \gamma_4\big( (a_1 \star a_3) \star a_2 + (a_1 \star a_2) \star a_3 \big), \\
			a_1 \star (a_2 &\star a_3) = \delta_1\big((a_1 \cdot a_3) \cdot a_2 - (a_1 \cdot a_2) \cdot a_3 \big) + \delta_2\big( (a_1 \star a_3) \cdot a_2 - (a_1 \star a_2) \cdot a_3 \big) \\
			&+ \delta_3\big( (a_1 \cdot a_3) \star a_2 - (a_1 \cdot a_2) \star a_3 \big) + \delta_4\big( (a_1 \star a_3) \star a_2 - (a_1 \star a_2) \star a_3 \big),
		\end{align*}
where the sixteen parameters $\alpha_1, \dots, \alpha_4, \beta_1, \dots, \beta_4, \gamma_1, \dots, \gamma_4, \delta_1, \dots, \delta_4$ belong to the ground field $\k$.
	
Let us write these equations in matrix form, where each column corresponds to a monomial, ordering them as in~\eqref{eq:order}.
	\begin{equation}\label{eq:matrix1}
		\left(
		\begin{array}{cccccccccccc}
			-1 & 0 & 0 & 0 & \alpha_1 & \alpha_1 & \alpha_2 & \alpha_2 & \alpha_3 & \alpha_3 & \alpha_4 & \alpha_4 \\
			0 & -1 & 0 & 0 & -\beta_1 & \beta_1 & -\beta_2 & \beta_2 & -\beta_3 & \beta_3 & -\beta_4 & \beta_4 \\
			0 & 0 & -1 & 0 & \delta_1 & \delta_1 & \delta_2 & \delta_2 & \delta_3 & \delta_3 & \delta_4 & \delta_4 \\
			0 & 0 & 0 & -1 & -\gamma_1 & \gamma_1 & -\gamma_2 & \gamma_2 & -\gamma_3 & \gamma_3 & -\gamma_4 & \gamma_4 \\
		\end{array}
		\right)
	\end{equation}
	
Since $\dim\calO(3)=8$, the consequences of our four identities obtained by the action of $S_3$ by permutations of arguments should be linear combinations of these identities themselves. We already ensured that the action of the transposition~$(2 3)$ preserves the vector space spanned by these identities. Since the group~$S_3$ is generated by the transpositions $(1 2)$ and $(2 3)$, it is enough to require that the linear span of the four identities is stable under the action of the transposition~$(1 2)$. That action transforms the rows of our matrix into 
	\[
		\left(
		\begin{array}{cccccccccccc}
			\alpha_1 & \alpha_2 & -\alpha_3 & -\alpha_4 & -1 & \alpha_1 & 0 & -\alpha_2 & 0 & \alpha_3 & 0 & -\alpha_4 \\
			-\beta_1 & -\beta_2 & \beta_3 & \beta_4 & 0 & \beta_1 & -1 & -\beta_2 & 0 & \beta_3 & 0 & -\beta_4 \\
			\delta_1 & \delta_2 & -\delta_3 & -\delta_4 & 0 & \delta_1 & 0 & -\delta_2 & 1 & \delta_3 & 0 & -\delta_4 \\
			-\gamma_1 & -\gamma_2 & \gamma_3 & \gamma_4 & 0 & \gamma_1 & 0 & -\gamma_2 & 0 & \gamma_3 & 1 & -\gamma_4 \\
		\end{array}
		\right) ,
	\]
so if $\dim\calO(3)=8$, the matrix 	
	\[
		\left(
		\begin{array}{cccccccccccc}
			-1 & 0 & 0 & 0 & \alpha_1 & \alpha_1 & \alpha_2 & \alpha_2 & \alpha_3 & \alpha_3 & \alpha_4 & \alpha_4 \\
			0 & -1 & 0 & 0 & -\beta_1 & \beta_1 & -\beta_2 & \beta_2 & -\beta_3 & \beta_3 & -\beta_4 & \beta_4 \\
			0 & 0 & -1 & 0 & \delta_1 & \delta_1 & \delta_2 & \delta_2 & \delta_3 & \delta_3 & \delta_4 & \delta_4 \\
			0 & 0 & 0 & -1 & -\gamma_1 & \gamma_1 & -\gamma_2 & \gamma_2 & -\gamma_3 & \gamma_3 & -\gamma_4 & \gamma_4 \\
			\alpha_1 & \alpha_2 & -\alpha_3 & -\alpha_4 & -1 & \alpha_1 & 0 & -\alpha_2 & 0 & \alpha_3 & 0 & -\alpha_4 \\
			-\beta_1 & -\beta_2 & \beta_3 & \beta_4 & 0 & \beta_1 & -1 & -\beta_2 & 0 & \beta_3 & 0 & -\beta_4 \\
			\delta_1 & \delta_2 & -\delta_3 & -\delta_4 & 0 & \delta_1 & 0 & -\delta_2 & 1 & \delta_3 & 0 & -\delta_4 \\
			-\gamma_1 & -\gamma_2 & \gamma_3 & \gamma_4 & 0 & \gamma_1 & 0 & -\gamma_2 & 0 & \gamma_3 & 1 & -\gamma_4 \\
		\end{array}
		\right)
	\]
must be of rank $4$. Performing elementary row operations to get a~$4 \times 4$ minor full of zeros in the bottom left corner of the matrix, we obtain in the bottom right corner a $4 \times 8$ minor with certain polynomials in~$\k[\alpha_1, \dots, \gamma_4]$ as entries (those polynomials are listed in Appendix~\ref{ap:equations3}). In order for the matrix to be of rank~$4$, all these polynomials have to vanish. The radical decomposition of the ideal formed by them can be computed using any computer algebra software such as \texttt{Magma}~\cite{MR1484478} or \texttt{SINGULAR}~\cite{DGPS}, and tells us that its geometric variety of solutions is formed by three irreducible components of dimensions 5, 4 and 0. This means that our set of 32 polynomials is not enough to finish the proof.
	
To proceed, we shall use the Gr\"obner basis condition. The Gr\"obner basis criterion furnished by Diamond Lemma \cite{MR3642294,MR2667136} asserts that a collection of elements forms a Gr\"obner basis if all common multiples of their leading terms admit an unambiguous rewriting into normal forms. 
In the proof of Proposition \ref{prop:BoundsMeet}, we already recalled that the mock-Lie identity does not form a Gr\"obner basis of the operad it defines, while the Jacobi identity in Lie algebras does form a Gr\"obner basis. This suggests that in the case under consideration, it is reasonable to look at the constraints on the parameters arising from the common multiple $a_1 \cdot (a_2 \cdot (a_3 \cdot a_4))$ of the leading term $a_1 \cdot (a_2 \cdot a_3)$ of the identity
\begin{align*}
			a_1 \cdot (a_2 &\cdot a_3) = \alpha_1\big((a_1 \cdot a_3) \cdot a_2 + (a_1 \cdot a_2) \cdot a_3 \big) + \alpha_2\big( (a_1 \star a_3) \cdot a_2 + (a_1 \star a_2) \cdot a_3 \big) \\
			&+ \alpha_3\big( (a_1 \cdot a_3) \star a_2 + (a_1 \cdot a_2) \star a_3 \big) + \alpha_4\big( (a_1 \star a_3) \star a_2 + (a_1 \star a_2) \star a_3 \big) \\
\end{align*}
with itself. 

Since it is a common multiple of the leading term $a_1 \cdot (a_2 \cdot a_3)$ with itself, \emph{a priori} there are two different ways to rewrite it. The first of them arises from the substitution $a_3\leftarrow a_3\cdot a_4$ into our identity. This way, we obtain 
	\begin{equation}\label{eq:firstDecomposition}
		\begin{aligned}
			a_1 \cdot (a_2 \cdot (a_3 \cdot a_4)) =& \alpha_1 (a_1\cdot (a_3\cdot a_4))\cdot a_2 + \alpha_1 (a_1\cdot a_2)\cdot (a_3\cdot a_4) \\
			{}&+ \alpha_2 (a_1\star (a_3\cdot a_4))\cdot a_2+ \alpha_2 (a_1\star a_2)\cdot (a_3\cdot a_4)\\ {}&+ \alpha_3 (a_1\cdot (a_3\cdot a_4))\star a_2 + \alpha_3 (a_1\cdot a_2)\star (a_3\cdot a_4) \\
			{}&+ \alpha_4 (a_1\star (a_3\cdot a_4))\star a_2 + \alpha_4 (a_1\star a_2)\star (a_3\cdot a_4).
		\end{aligned}
	\end{equation}
Note that we obtained a linear combination where two types of monomials appear: $(\_ * (\_ * \_))*\_ $ and $(\_ * \_) * (\_ * \_)$, where~$*$ can be either of the two operations. Each such monomial is divisible by a right comb, and therefore can be further rewritten. For instance, 
	\begin{equation}\label{eq:firstDecomposition_2}
		\begin{aligned}
			 (a_1\cdot (a_3\cdot a_4))\cdot a_2 =& \alpha_1 ((a_1\cdot a_4)\cdot a_3)\cdot a_2 + \alpha_1 ((a_1\cdot a_3)\cdot a_4)\cdot a_2\\ {} &+\alpha_2 ((a_1\star a_4)\cdot a_3)\cdot a_2+ \alpha_2 ((a_1\star a_3)\cdot a_4)\cdot a_2 \\
			 {} &+\alpha_3 ((a_1\cdot a_3)\star a_4)\cdot a_2+\alpha_3 ((a_1\cdot a_4)\star a_3)\cdot a_2 \\
			 {} &+ \alpha_4 ((a_1\star a_4)\star a_3)\cdot a_2 +\alpha_4 ((a_1\star a_3)\star a_4)\cdot a_2,
		\end{aligned}
	\end{equation}
	and
	\begin{equation}\label{eq:firstDecomposition_3}
		\begin{aligned}
			 (a_1\cdot a_2)\cdot (a_3\cdot a_4) =& \alpha_1 ((a_1\cdot a_2)\cdot a_4)\cdot a_3 + \alpha_1 ((a_1\cdot a_2)\cdot a_3)\cdot a_4\\
			 {} &+\alpha_2 ((a_1\cdot a_2)\star a_4)\cdot a_3 + \alpha_2 ((a_1\cdot a_2)\star a_3)\cdot a_4\\
			 {} &+\alpha_3 ((a_1\cdot a_2)\cdot a_4)\star a_3 +\alpha_3 ((a_1\cdot a_2)\cdot a_3)\star a_4 \\
			 {} &+ \alpha_4 ((a_1\cdot a_2)\star a_4)\star a_3+\alpha_4 ((a_1\cdot a_2)\star a_3)\star a_4.
		\end{aligned}
	\end{equation}
Performing this kind of rewriting for every monomial appearing in Equation~\eqref{eq:firstDecomposition}, we shall obtain a linear combination of left combs only.

On the other hand, rewriting the factor $a_2 \cdot (a_3 \cdot a_4)$ of our common multiple, we obtain
	\begin{equation}\label{eq:secondDecomposition}
		\begin{aligned}
			a_1 \cdot (a_2 \cdot (a_3 \cdot a_4)) =& \alpha_1 a_1\cdot ((a_2\cdot a_4)\cdot a_3) +\alpha_1 a_1\cdot ((a_2\cdot a_3)\cdot a_4) \\
			{} &+\alpha_2 a_1\cdot ((a_2\star a_4)\cdot a_3)+ \alpha_2 a_1\cdot ((a_2\star a_3)\cdot a_4)\\
			{} &+\alpha_3 a_1\cdot ((a_2\cdot a_4)\star a_3)+\alpha_3 a_1\cdot ((a_2\cdot a_3)\star a_4) \\
			{} &+\alpha_4 a_1\cdot ((a_2\star a_4)\star a_3)+\alpha_4 a_1\cdot ((a_2\star a_3)\star a_4).
		\end{aligned}
	\end{equation}
This way, we got a linear combination of monomials of the form $\_ \cdot ((\_ * \_) * \_)$, where~$*$ can be either of the two operations. Each such monomial is divisible by a right comb, and therefore can be further rewritten. That rewriting will not yet give a linear combination of left combs, as some elements of the form $(\_ * (\_ * \_))*\_ $ and $(\_ * \_) * (\_ * \_)$ may appear. Rewriting their right comb divisors, we shall obtain a linear combination of left combs.
	
Let us summarise the upshot of our calculation. Rewriting the monomial $a_1 \cdot (a_2 \cdot (a_3 \cdot a_4))$ in two possible ways, we obtain two different combinations of left combs. If our operad has a quadratic Gr\"obner basis for the reverse path-lexicographic ordering, the left combs must be linearly independent, so the two linear combinations we obtained must be equal. There are $48$ left combs of arity $4$, and thus we obtain $48$ new polynomial constraints on the values of the parameters $\alpha_1, \dots, \delta_4$. We already know that the maximal dimension of the irreducible component of the affine algebraic variety defined by constraints in arity~$3$ is equal to $5$. Thus, one may expect that taking just five of the $48$ equations should be sufficient for our purposes. This is indeed the case: if we look at the coefficients of the monomials 
	\[
		((a_1\cdot a_2)\cdot a_3)\cdot a_4, ((a_1\cdot a_2)\cdot a_3) \star a_4, ((a_1\cdot a_3)\cdot a_4)\cdot a_2, ((a_1\cdot a_3)\star a_4)\cdot a_2, ((a_1\cdot a_2)\star a_3)\cdot a_4,
	\]
that are listed in Appendix~\ref{ap:equations4}, joined with the 32 polynomials obtained on the previous step (listed in Appendix~\ref{ap:equations3}), we find that these polynomials have no common zeros. This follows from the fact (checked independently by several computer algebra systems, notably \texttt{Magma}~\cite{MR1484478} and \texttt{SINGULAR}~\cite{DGPS}) that the reduced Gr\"obner basis of the ideal generated by these polynomials for the lexicographical order of variables consists of the constant polynomial $1$. 
\end{proof}

\begin{theorem}\label{th:main}
The only non-trivial variety of non-associative algebras that is both a 2-variety and a Nielsen--Schreier variety is the variety of Lie algebras.
\end{theorem}

\begin{proof}
	We know that one of the three possibilities of Proposition~\ref{prop:BoundsMeet} must occur. The assumption on non-triviality of $\mathfrak{M}$ eliminates the first possibility, and Proposition~\ref{prop:generalCase} eliminates the third one. Thus, $\mathfrak{M}$ is the variety of Lie algebras.
\end{proof}

\section*{Funding} The first author was supported by Institut Universitaire de France, by Fellowship of the University of Strasbourg Institute for Advanced Study through the French national program ``Investment for the future'' (IdEx-Unistra, grant USIAS-2021-061), and by the French national research agency (grant ANR-20-CE40-0016). The second author was supported by Ministerio de Ciencia e Innovación (grant PID2021-127075NA-I00) and by a Postdoctoral Fellowship of the Research Foundation Flanders (FWO).

\section*{Acknowledgements} 
The second author would like to thank the Institut de Recherche Mathématique Avancée (IRMA) for its kind hospitality during his stay in Strasbourg.

\bibliographystyle{plain}
\bibliography{biblio}

\appendix
\section{Equations}
\subsection{Equations in degree 3}\label{ap:equations3}
\small
\allowdisplaybreaks
\begin{align*}
	f_{1}&= \alpha_1^2-\alpha_2 \beta_1+\alpha_4 \gamma_1-\alpha_3 \delta_1-1\\ 
	f_{2}&= \alpha_1^2+\alpha_1+\alpha_2 \beta_1-\alpha_4 \gamma_1-\alpha_3 \delta_1\\ 
	f_{3}&= \alpha_1 \alpha_2-\beta_2 \alpha_2+\alpha_4 \gamma_2-\alpha_3 \delta_2\\ 
	f_{4}&= \alpha_1 \alpha_2+\beta_2 \alpha_2-\alpha_2-\alpha_4 \gamma_2-\alpha_3 \delta_2\\ 
	f_{5}&= \alpha_1 \alpha_3-\delta_3 \alpha_3-\alpha_2 \beta_3+\alpha_4 \gamma_3\\ 
	f_{6}&= \alpha_1 \alpha_3-\delta_3 \alpha_3+\alpha_3+\alpha_2 \beta_3-\alpha_4 \gamma_3\\ 
	f_{7}&= \alpha_1 \alpha_4+\gamma_4 \alpha_4-\alpha_2 \beta_4-\alpha_3 \delta_4\\ 
	f_{8}&= \alpha_1 \alpha_4-\gamma_4 \alpha_4-\alpha_4+\alpha_2 \beta_4-\alpha_3 \delta_4\\ 
	f_{9}&= -\alpha_1 \beta_1+\beta_2 \beta_1-\beta_4 \gamma_1+\beta_3 \delta_1\\ 
	f_{10}&= -\alpha_1 \beta_1-\beta_2 \beta_1+\beta_1+\beta_4 \gamma_1+\beta_3 \delta_1\\ 
	f_{11}&= \beta_2^2-\alpha_2 \beta_1-\beta_4 \gamma_2+\beta_3 \delta_2-1\\ 
	f_{12}&= -\beta_2^2-\beta_2-\alpha_2 \beta_1+\beta_4 \gamma_2+\beta_3 \delta_2\\ 
	f_{13}&= -\alpha_3 \beta_1+\beta_2 \beta_3-\beta_4 \gamma_3+\beta_3 \delta_3\\ 
	f_{14}&= -\alpha_3 \beta_1-\beta_2 \beta_3+\beta_3+\beta_4 \gamma_3+\beta_3 \delta_3\\ 
	f_{15}&= -\alpha_4 \beta_1+\beta_2 \beta_4-\beta_4 \gamma_4+\beta_3 \delta_4\\ 
	f_{16}&= -\alpha_4 \beta_1-\beta_2 \beta_4-\beta_4+\beta_4 \gamma_4+\beta_3 \delta_4\\ 
	f_{17}&= \alpha_1 \delta_1-\delta_3 \delta_1-\beta_1 \delta_2+\gamma_1 \delta_4\\ 
	f_{18}&= \alpha_1 \delta_1-\delta_3 \delta_1+\delta_1+\beta_1 \delta_2-\gamma_1 \delta_4\\ 
	f_{19}&= \alpha_2 \delta_1-\beta_2 \delta_2-\delta_2 \delta_3+\gamma_2 \delta_4\\ 
	f_{20}&= \alpha_2 \delta_1+\beta_2 \delta_2-\delta_2-\delta_2 \delta_3-\gamma_2 \delta_4\\ 
	f_{21}&= -\delta_3^2+\alpha_3 \delta_1-\beta_3 \delta_2+\gamma_3 \delta_4+1\\ 
	f_{22}&= -\delta_3^2+\delta_3+\alpha_3 \delta_1+\beta_3 \delta_2-\gamma_3 \delta_4\\ 
	f_{23}&= \alpha_4 \delta_1-\beta_4 \delta_2+\gamma_4 \delta_4-\delta_3 \delta_4\\ 
	f_{24}&= \alpha_4 \delta_1+\beta_4 \delta_2-\gamma_4 \delta_4-\delta_3 \delta_4-\delta_4\\ 
	f_{25}&= -\alpha_1 \gamma_1-\gamma_4 \gamma_1+\beta_1 \gamma_2+\gamma_3 \delta_1\\ 
	f_{26}&= -\alpha_1 \gamma_1+\gamma_4 \gamma_1+\gamma_1-\beta_1 \gamma_2+\gamma_3 \delta_1\\ 
	f_{27}&= -\alpha_2 \gamma_1+\beta_2 \gamma_2-\gamma_2 \gamma_4+\gamma_3 \delta_2\\ 
	f_{28}&= -\alpha_2 \gamma_1-\beta_2 \gamma_2-\gamma_2+\gamma_2 \gamma_4+\gamma_3 \delta_2\\ 
	f_{29}&= -\alpha_3 \gamma_1+\beta_3 \gamma_2-\gamma_3 \gamma_4+\gamma_3 \delta_3\\ 
	f_{30}&= -\alpha_3 \gamma_1-\beta_3 \gamma_2+\gamma_3+\gamma_3 \gamma_4+\gamma_3 \delta_3\\ 
	f_{31}&= -\gamma_4^2-\alpha_4 \gamma_1+\beta_4 \gamma_2+\gamma_3 \delta_4+1\\ 
	f_{32}&= \gamma_4^2-\gamma_4-\alpha_4 \gamma_1-\beta_4 \gamma_2+\gamma_3 \delta_4
\end{align*}

\subsection{Equations in degree 4}\label{ap:equations4}
\small
\allowdisplaybreaks
\begin{align*}
	g_{1} =& \alpha_3 \delta_1+\alpha_1^2 -\gamma_1 \left(\alpha_4 \beta_2+\alpha_2^2\right)\\ &-\delta_1 \left(\alpha_3 \beta_2+\alpha_1 \alpha_2\right)-\alpha_1 \left(\alpha_3 \beta_1+\alpha_1^2\right)-\beta_1 \left(\alpha_4 \beta_1+\alpha_1 \alpha_2\right) \\
	g_{2} =& \alpha_3 \delta_3+\alpha_1 \alpha_3-\gamma_1 \left(\alpha_4 \beta_4+\alpha_2 \alpha_4\right) \\ & -\delta_1 \left(\alpha_3 \beta_4+\alpha_1 \alpha_4\right)-\alpha_1 \left(\alpha_3 \beta_3+\alpha_1 \alpha_3\right)-\beta_1 \left(\alpha_4 \beta_3+\alpha_2 \alpha_3\right) \\
	g_{3} =& \alpha_2 \delta_1+\alpha_1^2-\gamma_1 \left(\alpha_4 \beta_3-\alpha_2 \alpha_3\right)\\ &-\delta_1 \left(\alpha_1 \alpha_3-\alpha_3 \beta_3\right)-\alpha_1 \left(\alpha_1^2-\alpha_3 \beta_1\right)+\beta_1 \left(\alpha_1 \alpha_2-\alpha_4 \beta_1\right) \\
	g_{4} =& \alpha_1 \alpha_3-\gamma_2 \left(\alpha_4 \beta_3-\alpha_2 \alpha_3\right)\\ &-\delta_2 \left(\alpha_1 \alpha_3-\alpha_3 \beta_3\right)-\alpha_2 \left(\alpha_1^2-\alpha_3 \beta_1\right)-\beta_2 \left(\alpha_4 \beta_1-\alpha_1 \alpha_2\right)+\alpha_2 \delta_3 \\
	g_{5} =& \alpha_3 \delta_2+\alpha_1 \alpha_2-\gamma_3 \left(\alpha_4 \beta_2+\alpha_2^2\right)\\ &-\delta_3 \left(\alpha_3 \beta_2+\alpha_1 \alpha_2\right)-\alpha_3 \left(\alpha_3 \beta_1+\alpha_1^2\right)-\beta_3 \left(\alpha_4 \beta_1+\alpha_1 \alpha_2\right) 
\end{align*}

\end{document}